\documentclass[12pt,a4paper]{amsart}         
\usepackage[margin=3cm]{geometry}
\usepackage[backend=bibtex,url=false,doi=true,isbn=false,giveninits=true,style=numeric,sorting=nyt]{biblatex}
\usepackage{amsmath,amsfonts,amssymb,amsthm}
\usepackage{mathrsfs,epsfig,epstopdf,url}
\usepackage{enumerate}
\usepackage{xcolor}
\usepackage{float}
\usepackage{tikz}
\restylefloat{table}
\usepackage{multirow}
\usepackage{array}

\usepackage{pgfplots}

\newcommand{\C}{C_L(D,sP_{\infty})\mid_{\mathbb{F}_r}}
\theoremstyle{definition}
\newtheorem{definition}{Definition}%

\theoremstyle{plain}
\newtheorem{theorem}[definition]{Theorem}
\newtheorem{lemma}[definition]{Lemma}
\newtheorem{proposition}[definition]{Proposition}

\newtheorem{corollary}[definition]{Corollary}
\newtheorem{remark}[definition]{Remark}

\newcommand{\FF}{\mathbb{F}_8}

\DeclareMathOperator{\tr}{Tr}

\newcommand{\PG}{\mathsf{PG}}

\newcommand{\Ree}{\mathsf{Ree}}

\newcommand{\E}{\mathsf{E}}
\newcommand{\D}{\mathsf{D}}

\DeclareMathOperator{\Aut}{Aut}

\newcommand{\K}{\mathcal{K}}
\newcommand{\RRR}{\mathcal{R}(3)}

\title{Embeddings of Ree unitals in a projective plane over a field}

\author{G\'abor P. Nagy}
\address{Department of Algebra \\
	Budapest University of Technology and Economics\\
	Egry J\'ozsef utca 1\\
	H-1111 Budapest, Hungary}
\address{Bolyai Institute \\
	University of Szeged \\
	Aradi v\'ertan\'uk tere 1\\
	H-6720 Szeged, Hungary}
\email{nagyg@math.bme.hu}

\thanks{Support provided from the National Research, Development and Innovation Fund of Hungary, NKFIH-OTKA Grants 119687, 115288 and SNN 132625.}

\keywords{Ree unital, Ree group, projective embedding, admissible embedding}
\subjclass[2010]{51E20, 05B05}

\begin{document}

\begin{abstract}
We show that the Ree unital $\mathcal{R}(q)$ has an embedding in a projective plane over a field $F$ if and only if $q=3$ and $\mathbb{F}_8$ is a subfield of $F$. In this case, the embedding is unique up to projective linear transformations. Besides elementary calculations, our proof uses the classification of the maximal subgroups of the simple Ree groups. 
\end{abstract}

\maketitle

\section{Introduction}

A $t$-$(n, k, \lambda)$ \textit{design,} or equivalently a \textit{Steiner system} $S_\lambda(t,k,n)$, is a finite simple incidence structure consisting of $n$ points and a number of blocks, such that every block is incident with $k$ points and every $t$-subset of points is incident with exactly $\lambda$ blocks. Let $\mathcal{D}=(\mathcal{P},\mathcal{B},I)$ be a design and $\Pi=(\mathcal{P}',\mathcal{B}',I')$ a projective plane. The map $\varrho: \mathcal{P} \cup \mathcal{B} \to \mathcal{P}' \cup \mathcal{B}'$ is an \textit{embedding} of $\mathcal{U}$, provided it is injective, $\varrho(\mathcal{P})\subseteq \mathcal{P}'$, $\varrho(\mathcal{B})\subseteq\mathcal{B}'$, and 
\[\forall P\in \mathcal{P}, \forall B\in \mathcal{B}: \quad P I B \Leftrightarrow \varrho(P)I'\varrho(B).\]
The embedding $\varrho$ is \textit{admissible} (or \textit{equivariant}), if for any automorphism $\alpha$ of $\mathcal{U}$, there is a collineation $\beta$ of $\Pi$ such that $\varrho(P^\alpha)=\varrho(P)^\beta$ holds for all $P\in \mathcal{P}$. 

An \textit{abstract unital} or a \textit{unital design} of order $n$ is a $2$-$(n^3+1,n+1,1)$ design. The problem of the embeddings of abstract unitals in projective planes is a classical one with many old unsolved questions, see \cite{MR2795696,MR3533345,MR4008656,MR3963225,MR3692296,MR3812020,MR193136,MR2368994}. The classical \textit{hermitian unital} $\mathcal{H}(q)$ of order $q$ is constructed from the set of absolute points and non-absolute lines of the desarguesian plane $\PG(2,q^2)$. The abstract hermitian unital $\mathcal{H}(q)$ has a natural embedding in $\PG(2,q^2)$, which is unique up to projective equivalence, see \cite{MR3692296,MR4026645}. Moreover, this embedding is admissible. 

Another class of abstract unitals of order $q=3^{2n+1}$ was discovered by L\"uneburg \cite{MR193136}. Let $\Ree(q)={}^2G_2(q)$ be the \textit{Ree group} of order $(q^3+1)q^3(q-1)$, $q=3^{2n+1}$, see \cite{MR1777008,MR955589}. Then $\Ree(q)$ has a $2$-transitive action on $q^3+1$ points, namely by conjugation on the set of all Sylow $3$-subgroups. The pointwise stabilizer of two points $P,Q$ is cyclic of order $q-1$ and thus contains a unique involution $t$. It follows that $\Ree(q)$ has a unique conjugacy class of involutions, and any involution $t$ fixes exactly $q+1$ points. The blocks of the \textit{Ree unital} $\mathcal{R}(q)$ are the sets of fixed points of the involutions of $\Ree(q)$. $\mathcal{R}(q)$ admits the $\Ree(q)$ as a $2$-transitive automorphism group; the full automorphism group is larger, for $n\geq 1$, admitting also the field automorphism, see \cite{MR2795696}. The smallest Ree unital $\RRR$ and the smallest Ree group $\Ree(3)\cong\mathsf{P\Gamma{}L}(2,8) \cong \mathsf{PSL}(2,8)\rtimes C_3$ are a little different from the general case, see \cite{MR655065,MR847092}. For $q>3$, $\Ree(q)$ is simple.

L\"uneburg \cite{MR193136} showed that the Ree unital $\mathcal{R}(q)$ has no admissible embeddings in projective planes of order $q^2$ (desarguesian or not). For $q = 3$, Gr\"uning \cite{MR847092} proved that the smallest Ree unital $\RRR$ has no embedding in any projective plane of order $9$. Montinaro \cite{MR2368994} extended these results by showing that for $q\neq 3$ and $n\leq q^4$, $\mathcal{R}(q)$ has no admissible embedding in a projective plane of order $n$. Moreover, if $\RRR$ is embedded in a projective plane $\Pi$ of order $n\leq 3^4$ in an admissible way, then either $\Pi\cong \PG(2,8)$, or $n=2^6$, see \cite[Theorem 5]{MR2368994}. 

In this paper, we completely characterize the embeddings of $\RRR$ in a projective plane over a field, extending Montinaro's result. 
\begin{theorem} \label{thm:R3_embedding}
Let $F$ be a field and $\varphi:\RRR\to \PG(2,F)$ an embedding. Then the following hold:
\begin{enumerate}[(i)]
\item $\FF$ is a subfield of $F$, and the image of $\varphi$ is contained in a subplane of order $8$.
\item The embedding is unique up to $\Aut(\RRR)$ and $\mathsf{PGL}(3,F)$. 
\item The embedding is admissible.
\end{enumerate}
\end{theorem}
Our main result is the following:
\begin{theorem} \label{thm:noemb}
Let $n$ be a positive integer, and $q=3^{2n+1}$. Suppose that $\Pi$ is a projective plane such that for each embedding $\varphi:\RRR\to \Pi$, the image $\varphi(\RRR)$ is contained in a pappian subplane. Then the Ree unital $\mathcal{R}(q)$ has no embedding in $\Pi$. In particular, $\mathcal{R}(q)$ has no embedding in a projective plane over a field. 
\end{theorem}

These results suggest that the problem of projective embeddings of the Ree unitals can be reduced to the question whether the smallest Ree unital has an embedding in a non-desarguesian projective plane. This question is surprisingly hard, even if we assume that the embedding is admissible.

The structure of the paper is as follows. In section \ref{sec:prelim}, we present the embedding of $\RRR$ in $\PG(2,8)$, and some technical lemmas to ease the calculations in $\PG(2,8)$. In section \ref{sec:pentagons}, we study sets of five points of $\PG(2,8)$, determining ten external lines of a conic. Such \textit{external pentagons} correspond to \textit{super O'Nan configurations} of $\RRR$; their properties are listed in section \ref{sec:superonan}. In sections \ref{sec:R3emb} and \ref{sec:Rqnonemb}, we prove Theorems \ref{thm:R3_embedding} and \ref{thm:noemb}.

\section{Preliminaries} \label{sec:prelim}

The embedding of $\RRR$ in $\PG(2,8)$ deserves special attention. The construction was first given by Gr\"uning \cite{MR847092}, who attributes the idea to F.C. Piper. The embedding is slightly simpler to present in the dual setting. Let $\K$ be a non-singular conic in $\PG(2,8)$. The tangents of $\K$ have a common point $N$, which is called the \textit{nucleus} of $\K$ (see \cite{MR0333959}). The set $\mathcal{O}=\K\cup \{N\}$ is a hyperoval, that is, a set of $10$ points such that each line intersects it in $0$ or $2$ points. The point $P$ is \textit{external}, if $P\not\in \mathcal{O}$. The line $\ell$ is \textit{external,} if $\ell\cap \mathcal{O}= \emptyset$. There are $63$ external points, $28$ external lines, and each external point is incident with $4$ external lines. In other words, the external points and the external lines form a (dual) unital $\mathcal{U}$ of order $3$. Let $G=\mathsf{P{\Gamma}O}(3,8)$ be the group of projective semilinear transformations of $\PG(2,8)$, preserving $\mathcal{O}$. $G$ is isomorphic to 
\[\mathsf{P{\Gamma}L}(2,8) \cong \mathsf{PSL}(2,8) \rtimes C_3,\]
and acts $2$-transitively on the set of external lines. Hence, $\mathcal{U}$ has a $2$-transitive automorphism group and $\RRR \cong \mathcal{U}$ by \cite{MR655065}. We call the isomorphism $\varphi:\RRR \to \mathcal{U}$ a \textit{dual embedding} of $\RRR$ into $\PG(2,8)$ with respect to the conic $\mathcal{K}$. 

To make the computation in $\FF$ more transparent, we fix a root $\gamma\in\FF$ of the polynomial $X^3+X+1=0$ in $\FF$. Then
\[\gamma^4=\gamma^2+\gamma, \quad
\gamma^5=\gamma^2+\gamma+1, \quad
\gamma^6=\gamma^2+1, \quad
\gamma^7=1.\] 
The trace map of $\mathbb{F}_8$ over $\mathbb{F}_2$ is defined as
\[\tr(x)=x+x^2+x^4.\]
We fix the coordinate frame $(X,Y,Z)$ in $\PG(2,8)$ and extend the action of the Frobenius automorphism $\Phi:x\mapsto x^2$ to the points and lines of $\PG(2,8)$. In this way, we obtain a projective semilinear transformation of order $3$. For $c\in \FF$, the map 
\[\tau_c:(x,y,z) \to (x+cz,y+c^2z,z)\]
is an elation with axis $Z=0$.
\begin{lemma} \label{lm:Gamma}
\begin{enumerate}[(i)]
\item If $c\neq 0$ then the line $Y=mX+bZ$ is $\tau_c$-invariant if and only if $m=c$.
\item $\Phi$ and $\tau_c$ ($c\in \FF$) preserve the conic $\K:X^2+YZ=0$ of $\PG(2,8)$  and the line $\ell_\infty:Z=0$ at infinity.
\item The line $Y=mX+bZ$ is external to $\K$ if and only if $\tr(b/m^2)=1$. 
\item Let $\Gamma$ denote the group
\[\Gamma = \{\tau_c \mid \tr(c)=0\}\]
of elations. $\Gamma$ is elementary abelian of order $4$. By conjugation, $\Phi$ permutes the nontrivial elements of $\Gamma$. 
\item The group $G_0=\langle \Gamma, \Phi \rangle$ of semilinear transformations is isomorphic to $A_4$. 
\end{enumerate}
\end{lemma}
\begin{proof}
$\tau_c$ maps $Y=mX+bZ$ to $Y=mX+(b+c^2+cm)Z$. This implies (i). (ii) is trivial. (iii) follows from the fact that in a finite field $\mathbb{F}_q$ of even order, the quadratic form $X^2+mX+b$ is reducible if and only if $\tr(b/m^2)=0$. (iv) and (v) are immediate. 
\end{proof}

Finally, we present a useful elementary result on groups acting on graphs:
\begin{lemma} \label{lm:connectedcomponent}
Let $G$ be a group acting primitively on the set of vertices of the graph $\Gamma$. Then $\Gamma$ is either empty or connected.
\end{lemma}
\begin{proof}
The connected components of $\Gamma$ are blocks of imprimitivity of $G$. 
\end{proof}

\section{External pentagons in $\PG(2,8)$} \label{sec:pentagons}

Let $p$ be a prime, and $q=p^e$ be a prime power. Let $\K$ be a non-singular conic in $\PG(2,q)$. For any line $\ell$, we have $|\K\cap \ell|\leq 2$. We call $\ell$ \textit{secant,} \textit{tangent} or \textit{external} to $\K$, according if $|\K\cap \ell|$ is $2$, $1$, or $0$. If $q$ is even, then all tangents pass through a common point, the \textit{nucleus} of $\K$. (See \cite{MR0333959}.) The group of collineations  preserving $\K$ is $\mathsf{P{\Gamma}O}(3,q)$. One has the isomorphisms $\mathsf{PGO}(3,q) \cong \mathsf{PGL}(2,q)$ and $\mathsf{P{\Gamma}O}(3,q) \cong \mathsf{P{\Gamma}L}(2,q)$. 

\begin{definition}
Let $\K$ be a conic in $\PG(2,q)$. We say that the points $P_0,\ldots,P_4$ in general position form an \textit{external pentagon} with respect to $\K$, if $P_iP_j$ are external lines of $\K$, $0\leq i<j\leq 4$. The external pentagon is said to have type $A_4$, if there is a group $G_0$ of collineations preserving $\K$ and $\{P_0,\ldots,P_4\}$, such that $G_0\cong A_4$. 
\end{definition}

Notice that the points of an external pentagon are not in $\mathcal{K}$, and if $q$ is even, then they are also distinct from the nucleus of $\mathcal{K}$.

\begin{lemma}
Let $\mathcal{P}=\{P_0,\ldots,P_4\}$ be an external pentagon of type $A_4$ in $\PG(2,q)$ with respect to the conic $\K$. Then $q$ is even and the following hold:
\begin{enumerate}[(i)]
\item $G_0$ fixes one point of $\mathcal{P}$ and acts $2$-transitively on the remaining four. 
\item If $q=8$, then $G_0\cong A_4$ is the full stabilizer of the sets $\K$ and $\mathcal{P}$ in the collineation group of $\PG(2,8)$. 
\end{enumerate}
\end{lemma}
\begin{proof}
(i) The Sylow $2$-subgroup $T=G_0'$ of $G_0$ is normal in $G_0$, and $T\leq \mathsf{PGO}(3,q)$. Hence, $T$ acts faithfully on $\mathcal{P}$, with a unique fixed point, say, $P_0$. As $T$ is normal in $G_0$, $G_0$ fixes $P_0$. Moreover, $T$ acts regularly and $G_0$ acts $2$-transitively on $\{P_1,\ldots,P_4\}$. Let $H$ be the stabilizer of $P_0$ in $\mathsf{P{\Gamma}O}(3,q)$. If $q$ is odd then either $H\cong D_{2(q+1)} \rtimes C_e$ (if no tangent of $\mathcal{K}$ is incident with $P$), or $H\cong D_{2(q-1)} \rtimes C_e$ (if $2$ tangents are incident with $P$). In both cases, the dihedral subgroup $D_{2(q\pm 1)}$ contains a unique central involution $\alpha$, which is contained in each elementary abelian subgroup of order $4$. Hence, no subgroup of $H$ of order $12$ can be isomorphic to $A_4$, a contradiction. 

(ii) Assume $q=8$. Then $H\cong \mathbb{F}_8^+\rtimes C_3$ has order $24$ and we have to show that $H$ does not leave $\mathcal{P}$ invariant. Let $t_0$ be the tangent line through $P_0$ to $\K$. The elementary abelian part $A$ of $H$ consists of elations with respect to $t_0$, that is, for any $2$-element $\alpha \in A$, the set of fixed point of $\alpha$ is $t_0$. As for $i\geq 1$, $P_i\not\in t_0$, and the $A$-orbit of $P_i$ has length $8$. 
\end{proof}

\begin{definition}
The \textit{fundamental pentagon} $\mathcal{F}$ of $\PG(2,8)$ is the set of points
 $A(1,1,0)$, $C_1(0,1,1)$ and
\[C_2=(\gamma,\gamma^6,1), \quad C_3=(\gamma^2,\gamma^5,1), \quad C_4=(\gamma^4,\gamma^3,1).\]
\end{definition}
\begin{lemma}\label{lm:fundamental}
The fundamental pentagon $\mathcal{F}$ is an external pentagon with respect to the conic $\K:X^2+YZ=0$. Moreover, $\mathcal{F}$ is of $A_4$ type with collineation group $G_0=\langle \Gamma, \Phi \rangle$.
\end{lemma}
\begin{proof}
The following facts can be checked by calculations: $A$ is fixed by $G_0$. $\Phi$ fixes $C_1$ and permutes $C_1,C_2,C_3$. The $\Gamma$-orbit of $C_1$ is $\{C_1,\ldots,C_4\}$. The lines $AC_1:Y=X+Z$ and $C_1C_2:Y=\gamma X+Z$ are external.
\end{proof}
\begin{lemma} \label{lm:penta_unique}
Let $\K$ be a conic in $\PG(2,8)$ and $\mathcal{P}$ an external pentagon of type $A_4$ with respect to $\K$. The projective coordinate frame can be chosen such that $\K$ has equation $X^2+YZ=0$ and $\mathcal{P}$ is the fundamental pentagon.
\end{lemma}
\begin{proof}
We can assume the equation $\K:X^2+YZ=0$ and that $A(1,1,0)$ is the point of $\mathcal{P}$ which is fixed by $G_0$. Then $G_0$ is a subgroup of 
\[H=\{\tau_c \mid c \in \FF\} \rtimes \langle \Phi \rangle,\]
which is the stabilizer of $\K$ and $A$. The $2$-subgroup $\{\tau_c \mid c \in \FF\}$ has two $\Phi$-invariant, irreducible proper subgroups: $Z(H)=\langle \tau_1 \rangle$ and $\Gamma$. Hence, $\langle\Gamma,\Phi\rangle$ is the unique subgroup $H$ which is isomorphic to $A_4$, and $G_0=\langle\Gamma,\Phi\rangle$ follows. Let $C_1$ denote the point of $\mathcal{P}\setminus \{A\}$ that is fixed by $\Phi$. As $AC_1$ is an external line, $C_1$ must have coordinates $(x,x+b,1)$ with $x,b\in \mathbb{F}_2$ and $\tr(b)=1$. This means that either $C_1=(0,1,1)$ or $C_1=(1,0,1)$. Applying $\tau_1$ to $\mathcal{P}$, we can assume $C_1=(0,1,1)$. Straightforward computation shows that $\{C_1,\ldots,C_4\}$ is the $\Gamma$-orbit of $C_1$, and $\{A,C_1,\ldots,C_4\}$ is indeed the funtamental pentagon.
\end{proof}
\begin{remark} \label{rem:126}
Lemma \ref{lm:Gamma}(ii) and Lemma \ref{lm:penta_unique} imply that with fixed conic $\K$ of $\PG(2,8)$, the number of $A_4$-type external pentagons is $|\mathsf{P{\Gamma}O}(3,8):G_0|=126$. 
\end{remark}

For the rest of this section, we use the notation of Lemma \ref{lm:Gamma} for $\K$, $\Gamma$, $G_0$, and $\tau_c$. 
\begin{proposition} \label{pr:penta}
Let $\mathcal{F}=\{A,C_1,\ldots,C_4\}$ be the fundamental pentagon in $\PG(2,8)$. For any even permutation $ijk\ell$ of $\{1,2,3,4\}$, let $d_{ijk\ell}$ denote the line connecting the points $AC_i\cap C_jC_\ell$ and $AC_j\cap C_kC_\ell$. The following hold:
\begin{enumerate}[(i)]
\item $G_0$ permutes the lines $d_{ijk\ell}$ regularly. In particular, the lines $d_{ijk\ell}$ are distinct. 
\item The lines $d_{ijk\ell}$ are external to $\K$.
\item For any coset $\Gamma g$ of $\Gamma$, the four lines $d_\pi$, $\pi \in \Gamma g$, share a common point at infinity $Z=0$. 
\item The lines $AC_4$, $d_{1234}$ and $d_{3241}$ are concurrent.
\end{enumerate}
\end{proposition}
\begin{proof}
Obviously, $G_0$ acts on the lines $d_{ijk\ell}$. We have
\[AC_1\cap C_2C_4 =(\gamma^6,\gamma^2,1), \quad AC_2\cap C_3C_4 =(1,\gamma^4,1),\]
with connecting line $d_{1234}:Y=\gamma^6 X+\gamma^3 Z$. The intersection $d_{1234}\cap \ell_\infty=(1,\gamma^6,0)$ is not fixed by any element of order $3$ of $G$. Thus, the stabilizer of $d_{1234}$ in $G$ is contained in $\Gamma$. Since $\tr(\gamma^6)=1$, the stabilizer of $d_{1234}$ in $\Gamma$ is trivial by Lemma \ref{lm:Gamma}(i). This proves (i), and also (ii), since $d_{1234}$ is an external line. (iii) follows from the fact that $\Gamma$ fixes the points at infinity. Computing the equations and the determinant
\[\det\begin{bmatrix}
1&1&\gamma^6 \\ \gamma^6&1&\gamma^3 \\ \gamma^3&1&\gamma^4
\end{bmatrix} =0,\]
we obtain (iv). 
\end{proof}

\section{Super O'Nan configurations in $\RRR$} \label{sec:superonan}

In a $2$-design, an \textit{O'Nan (or Pasch) configuration} consists of four pairwise intersecting blocks, no three of which pass through the same point. Brouwer \cite{MR655065} observed that in $\RRR$, each O'Nan configuration is contained in a \textit{super O'Nan configuration,} that is, in a set of five pairwise intersecting blocks in general position. In this section, we collect some facts on super O'Nan configurations of $\RRR$. 

\begin{lemma}[{\cite{MR655065}}] \label{lm:nrSOCs}
The number of super O'Nan configurations in $\RRR$ is $126$. \qed
\end{lemma}

For the rest of this section, we fix a dual embedding $\varphi^*$ of $\RRR$ in $\PG(2,8)$ with respect to the conic $\K:X^2+YZ=0$. Notice that the blocks $a,b$ of $\RRR$ intersect if and only if the points $\varphi^*(a),\varphi^*(b)$ determine an external line of $\K$. This implies that $\{b_0,\ldots,b_4\}$ is a super O'Nan configuration of $\RRR$ if and only if $\{\varphi^*(b_0),\ldots,\varphi^*(b_4)\}$ is an external pentagon with respect to $\K$. 
\begin{lemma} \label{lm:autSOC}
$\Ree(3)$ acts transitively on the set of super O'Nan configurations of $\RRR$. The stabilizer of a super O'Nan configuration is isomorphic to $A_4$. It fixes one of the blocks $b_i$ and acts $2$-transitively on the remaining four. 
\end{lemma}
\begin{proof}
Remark \ref{rem:126} and Lemma \ref{lm:nrSOCs} imply that each external pentagon of $\PG(2,8)$ is of $A_4$ type. Our claim follows from Lemma \ref{lm:penta_unique}. 
\end{proof}

\begin{proposition} \label{pr:mainSOC}
Let $\mathcal{B}$ be a super O'Nan configuration of $\RRR$ with stabilizer subgroup $S\cong A_4$. We can label the blocks of $\mathcal{B}$ by $a,c_1,\ldots,c_4$ such that for any even permutation $ijk\ell$, the blocks $(a\cap c_i)(c_j\cap c_\ell)$ and $(a\cap c_j)(c_k\cap c_\ell)$ have a unique intersection $D_{ijk\ell}$. Moreover, the following hold:
\begin{enumerate}[(i)]
\item The points $D_{ijk\ell}$ are distinct. 
\item Let $T$ be the Sylow $2$-subgroup of $S$. For any coset $T g$ of $S$, the four points $D_\pi$, $\pi \in T g$, form a block. 
\item The points $a\cap c_4$, $D_{1234}$ and $D_{3241}$ are contained in a block.
\end{enumerate}
\end{proposition}
\begin{proof}
We use Lemma \ref{lm:autSOC}, Proposition \ref{pr:penta} and the dual embedding $\varphi^*$ of $\RRR$ in $\PG(2,8)$. 
\end{proof}

\section{Embeddings of $\RRR$} \label{sec:R3emb}

\begin{proof}[Proof of Theorem \ref{thm:R3_embedding}]
Let $\{a,c_1,\ldots,c_4\}$ be a super O'Nan configuration of $\RRR$, as given in Proposition \ref{pr:mainSOC}. Let us choose the projective coordinate frame of $\PG(2,F)$ such that $\varphi(c_1): X+Y+Z=0$, $\varphi(c_2): X=0$, $\varphi(c_3): Y=0$, $\varphi(c_4): Z=0$. There are elements $u,v\in F\setminus\{0,1\}$, $u\neq v$, such that $\varphi(a):uX+vY+Z=0$. We have
\begin{align*}
\varphi(D_{1234})&=(v^2-v,v-u,v^2-uv), & \varphi(D_{2143})&=(u-uv,u-1,v-u), \\ 
\varphi(D_{3412})&=(v,uv-u,-uv), & \varphi(D_{4321})&=(v-u,u^2-u,u-uv).
\end{align*}
By Proposition \ref{pr:mainSOC}(ii), these are collinear points, thus,
\[\det\begin{bmatrix}
v^2-v&v-u&v^2-uv \\ u-uv&u-1&v-u \\ v&uv-u&-uv
\end{bmatrix} = u(u-1)v(v-1)(v^2-uv+2u-3v)=0,\]
and
\[\det\begin{bmatrix}
v^2-v&v-u&v^2-uv \\ u-uv&u-1&v-u \\ v-u&u^2-u&u-uv
\end{bmatrix} = u(u-1)v(v-1)(uv-u^2+2u-3v+1)=0.\]
The difference of these two equations is
\[u(u-1)v(v-1)((u-v)^2-1)=0,\]
which implies $u=v\pm 1$. Substituting back, we obtain either $2v-2=0$ or $2=0$, which are not possible unless $\mathrm{char}(F)=2$. In this case, all equations so far reduce to $u+v+1=0$. Computing $\varphi(a\cap c_4)=(v,v+1,0)$ and
\[\varphi(D_{3241}) = (v+1,1,v^2),\]
we have
\[\det\begin{bmatrix}
v&v+1&0 \\ v^2+v&1&v \\ v+1&1&v^2
\end{bmatrix} = v(v+1)(v^3+v^2+1)=0.\]
This shows that $u,v\in \FF$, and for any even permutation $ijk\ell$, $\varphi(D_{ijk\ell})$ is contained in the subplane $\PG(2,8)$. Hence, at least $22$ points of $\varphi(\RRR)$ are contained in $\PG(2,8)$. If $Q$ is one of the remaining $6$ points, then there are at least two blocks through $Q$ with equation over $\FF$, and therefore $Q$ is in $\PG(2,8)$ as well. The computation shows that up to the action of the Frobenius map $\Phi$, the embedding $\varphi$ is uniquely determined by the images of the blocks $c_1,\ldots,c_4$. In particular, $\varphi$ must be an embedding with respect to a dual conic $\K^*$. All subplanes of order $8$, and all dual conics of a given subplane of order $8$ are projectively equivalent in $\PG(2,F)$. Hence, we obtain (ii). Montinaro's \cite[Theorem 5]{MR2368994} implies (iii).
\end{proof}

\begin{corollary} \label{coro:concurrent}
Let $F$ be a field and $\varphi:\RRR\to \PG(2,F)$ an embedding. Let $S=\{1,a_1,\ldots,a_7\}$ be a Sylow $2$-subgroup of $\Ree(3)$. Then the lines $\varphi(a_1), \ldots,\varphi(a_7)$ are concurrent. 
\end{corollary}
\begin{proof}
Consider the dual embedding $\varphi^*:\RRR \to \PG(2,F)$. By Theorem \ref{thm:R3_embedding}, $\varphi^*(\RRR)$ is contained in a subplane $\PG(2,8)$. As before, we identify the blocks of $\RRR$ and the involutions of $\Ree(3)$. For a block $a$, $\varphi^*(a)$ is an external point of the conic $\mathcal{K}$. Moreover, $a$ determines a unique collineation $\hat{a}$ of order $2$. The following are equivalent:
\begin{enumerate}[(1)]
\item Two involutions $a_1,a_2\in \Ree(3)$ commute.
\item The involutions $\hat{a}_1, \hat{a}_2\in \mathsf{PGL}(3,F)$ commute.
\item The line $\varphi^*(a_1)\varphi^*(a_2)$ is tangent to $\mathcal{K}$. 
\item $\hat{a}_1$, $\hat{a}_2$ fix the same point of $\mathcal{K}$. 
\end{enumerate}
This implies that $\varphi^*(a_1), \ldots,\varphi^*(a_7)$ are contained in a tangent of $\mathcal{K}$, which is our claim in the dual setting. 
\end{proof}

\section{The nonexistence of embeddings of $\mathcal{R}(q)$, $q\geq 27$} \label{sec:Rqnonemb}

In this section, we write $q=3^{2n+1}$, and $G=\Ree(q)$. We have $2n+1=|\mathrm{Out}(\Ree(q))|$ and for any divisor $\alpha$ of $2n+1$ there is an outer automorphism $\psi_\alpha$ of $\Ree(q)$ of order $\alpha$. (See \cite[Lemma 6.2]{MR955589}.) Write $q_0=q^{1/\alpha}=3^{2n_0+1}$ and $G_0=C_{\Ree(q)}(\psi_\alpha)$. We have $G_0\cong \Ree(q_0)$. 

In order to be self-contained, we recall Kleidman's classification \cite[Theorem C]{MR955589} of the maximal subgroups of $G$, see also \cite{MR3924767}. If $q\geq 27$ and $H$ is a maximal subgroup of $G$, then one of the following cases occurs:
\begin{enumerate}[(M1)]
\item $H$ is a $1$-point stabilizer, isomorphic to the semidirect product of a group of order $q^3$ with the cyclic group of order $q-1$.
\item $H\cong \Ree(q_0)$, where $q_0=q^{1/\alpha}$, $\alpha$ prime. 
\item $H\cong C_2\times \mathsf{PSL}(2,q)$ is the centralizer of an involution.
\item $H\cong (C_2^2\times D_{(q+1)/2}) \rtimes C_3$ is the normalizer of a subgroup of order $4$.
\item $H\cong C_{q+\sqrt{3q}+1} \rtimes C_6$.
\item $H\cong C_{q-\sqrt{3q}+1} \rtimes C_6$.
\end{enumerate}
If $q=3$ then (M2) and (M6) do not occur. Moreover, $H\cong \mathsf{PSL}(2,8)$, or $H\cong (C_2^3\rtimes C_7)\rtimes C_3$ is the normalizer of a Sylow $2$-subgroup, that contains the subgroups (M3) and (M4). 

In $\Ree(q)$, the stabilizer of two points is cyclic of order $q-1$. Hence, the intersection of two Sylow $3$-subgroups is trivial. This implies that any Sylow $3$-subgroup $S_0$ of $G_0$ is contained in a unique Sylow $3$-subgroup $S$ of $G$, and $S$ is left invariant by $\psi_\alpha$. Conversely, let $S$ be a $\psi_\alpha$-invariant Sylow $3$-subgroup of $G$. The normalizer $N_G(S)$ is a parabolic subgroup of $G$, isomorphic to the semidirect product of a group of order $q^3$ with the cyclic group of order $q-1$. The centralizer of the field automorphism in $N_G(S)$ has order $q_0^3(q_0-1)$. This shows that $S_0=S\cap G_0$ is a Sylow $3$-subgroup in $G_0$. 

\begin{proposition} \label{pr:subdesign}
Let $q=3^{2n+1}$, $q_0=3^{2n_0+1}$ such that $2n_0+1$ divides $2n+1$. Then $\mathcal{R}(q)$ has a subdesign $\mathcal{D}\cong \mathcal{R}(q_0)$. Moreover, the stabilizer of $\mathcal{D}$ in $\Ree(q)$ is isomorphic to $\mathsf{Ree}(q_0)$. In particular, $\RRR$ is a subdesign of $\mathcal{R}(q)$ with stabilizer subgroup $\Ree(3)$. 
\end{proposition}
\begin{proof}
Remember that the points and blocks of $\mathcal{R}(q)$ can be identified with the Sylow $3$-subgroups, and the involutions of $\Ree(q)$, respectively. Hence, any automorphism of $G=\Ree(q)$ induces an automorphism of $\mathcal{R}(q)$. The involutions fixed by $\psi_\alpha$ and the Sylow $3$-subgroups left invariant by $\psi_\alpha$ form a subdesign $\mathcal{D}$ of $\mathcal{R}(q)$. As explained above, $\psi_\alpha$-invariant involutions and Sylow $3$-subgroups of $G$ correspond to involutions and Sylow $3$-subgroups of $G_0$. Hence, $\mathcal{D} \cong \mathcal{R}(q_0)$. Let $T_0$ be the stabilizer of $\mathcal{D}$ in $G$; clearly $G_0\leq T_0$. Looking at the list of maximal subgroups of $G$ in \cite[Theorem C]{MR955589}, we see that either $T_0=\Ree(q)$, or $T_0$ is contained in a subgroup isomorphic to $\Ree(q_1)$ with $q_1=q^{1/\beta}$, $\beta$ prime. Repeating this argument, we conclude that $T_0$ itself is isomorphic to a Ree group $\Ree(q_*)$, where $\mathbb{F}_{q_*}$ is a subfield of $\mathbb{F}_q$. As $T_0$ preserves the set of involutions of $G_0$, the only possibility is $q_0=q_*$. 
\end{proof}

We are now in the position to prove Theorem \ref{thm:noemb}.

\begin{proof}[Proof of Theorem \ref{thm:noemb}]
Let us suppose that an embedding $\varphi:\mathcal{R}(q)\to \Pi$ exists. Let $I$ denote the set of involutions of $\Ree(q)$. In three steps, we show that all lines $\varphi(a)$ ($a\in I$) are concurrent, a contradiction.

Claim 1: For any Sylow $2$-subgroup $S=\{1,a_1,\ldots,a_7\}$ of $\Ree(q)$, the lines $\varphi(a_1), \ldots,\varphi(a_7)$ are concurrent. 

By Proposition \ref{pr:subdesign}, $\mathcal{R}(q)$ has a subdesign $\mathcal{D}\cong \RRR$, whose stabilizer is the subgroup $\Ree(3)$ of $\Ree(q)$. Both $\Ree(q)$ and $\Ree(3)$ have an elementary abelian Sylow $2$-subgroup of order $8$. Hence, we can assume w.l.o.g. that $S\leq \Ree(3)$. By the assumption, $\varphi(\mathcal{D})$ is contained in a pappian subplane $\Pi_0$. We apply Corollary \ref{coro:concurrent} to the restriction $\varphi:\mathcal{D}\to \Pi_0$ to prove the claim. 

Claim 2: For any $a\in I$, there is a point $P_a\in \varphi(a)$ of $\PG(2,F)$ such that $P_a\in \varphi(b)$ for all $b\in I$ with $ab=ba$. 

Fix $a\in I$. The centralizer $C_G(a)$ is $\langle a \rangle \times T$, with $T\cong \mathsf{PSL}(2,q)$. $T$ has a unique class $J$ of involutions. For arbitrary \textit{commuting involutions} $b,c\in J$, $\langle a,b,c \rangle$ is contained in a Sylow $2$-subgroup of $G$. By claim 1, the lines $\varphi(b),\varphi(c)$ intersect on $\varphi(a)$. Fix $b\in J$ and define $P_a=\varphi(a) \cap \varphi(b)$. Let $\Gamma$ be the graph with vertex set $J$ and $c_1,c_2\in J$ connected by an edge if and only if $c_1c_2=c_2c_1$. As $q\geq 27$ odd, $C_T(b)\cong D_{q+1}$ is maximal in $T$ by Dickson's Theorem. $T$ acts primitively on $J$ and Lemma \ref{lm:connectedcomponent} implies that $\Gamma$ is connected. Hence, for \textit{any} $c\in J$, there are elements $b_0=b,b_1,\ldots,b_k=c\in J$ such that $b_ib_{i+1}=b_{i+1}b_i$ for all $i=0,\ldots,k-1$. For all indices $i$, $\varphi(a)\cap \varphi(b_i)=\varphi(a)\cap \varphi(b_{i+1})$. Hence, $\varphi(a)$, $\varphi(b)$ and $\varphi(c)$ are concurrent. If $c$ is an involution of $C_G(a)$, not in $J\cup \{a\}$, then $ac\in J$ and $\varphi(a)$, $\varphi(b)$ and $\varphi(ac)$ are concurrent. Also, the lines $\varphi(a)$, $\varphi(c)$ and $\varphi(ac)$ are concurrent, that shows $P_a\in \varphi(c)$. 

Claim 3: All lines $\varphi(a)$ ($a\in I$) of the embedding are concurrent. 

If $a,b\in I$ commute then $P_a=P_b$. As above, construct the graph $\Gamma'$ with vertex set $I$ and $c_1,c_2\in I$ connected by an edge if and only if $c_1c_2=c_2c_1$. By \cite[Theorem C]{MR955589}, $G$ acts primitively on $I$, and $\Gamma'$ is connected by Lemma \ref{lm:connectedcomponent}. Fix arbitrary elements $a,b\in I$.  There are elements $a_0=a, a_1,\ldots,a_k=b\in I$ such that $a_ia_{i+1}=a_{i+1}a_i$ for all $i=0,\ldots,k-1$. Then $P_{a}=P_{a_0}=\ldots=P_{a_k}=P_b$, that finishes the proof.
\end{proof}

\printbibliography

\end{document}